\documentclass[12pt]{article}

\usepackage{amssymb,amsmath,amsfonts,amsthm}

\usepackage[dvips]{graphics,color}
\usepackage{graphicx}
\usepackage{url}

\newtheorem{theorem}{Theorem}[section]
\newtheorem{lemma}[theorem]{Lemma}

\newtheorem{remark}[theorem]{Remark}

\newtheorem{corollary}[theorem]{Corollary}

\usepackage{endnotes}

\usepackage[T1]{fontenc}

\usepackage{textcomp}

\begin{document}

\title{The $n$-Queens Problem in Higher Dimensions}
\author{\begin{tabular}[t]{c@{\extracolsep{2em}}c}
    Jeremiah Barr & Shrisha Rao\footnote{Corresponding author} \\
    Dept of Computer Science & IIIT-Bangalore\\
    Mount Mercy College & 26/C, Electronics City\\
    Cedar Rapids, IA 52402 & Bangalore 560 100\\
    U.~S.~A. & India\\
    jrb62307@att.net & srao@iiitb.ac.in
\end{tabular}}
\maketitle

\begin{abstract}
A well-known chessboard problem is that of placing eight queens on
the chessboard so that no two queens are able to attack each
other. (Recall that a queen can attack anything on the same row,
column, or diagonal as itself.)  This problem is known to have
been studied by Gauss, and can be generalized to an \(n \times n\)
board, where \(n \geq 4\).  We consider this problem in
$d$-dimensional chess spaces, where \(d \geq 3\), and obtain the
result that in higher dimensions, $n$ queens do not always suffice
(in any arrangement) to attack all board positions.  Our methods
allow us to obtain the first lower bound on the number of queens
that are necessary to attack all positions in a $d$-dimensional
chess space of size $n$, and further to show that for any $k$,
there are higher-dimensional chess spaces in which not all
positions can be attacked by \(n^k\) queens.

\end{abstract}

\section{Introduction} \label{intro}

The 8-queens problem is a well-known chessboard problem, whose
constraints are to place eight queens on a normal chessboard in such a
way that no two attack each other, under the rule that a chess queen
can attack other pieces in the same column, row, or diagonal.  This
problem can be generalized to place $n$ queens on an $n$ by $n$
chessboard, otherwise known as the $n$-queens problem.  The
mathematicians Gauss and Polya studied this problem~\cite{3}, and
Ahrens~\cite{1} showed that for all \(n \geq 4\), solutions
exist. This problem can be further generalized to $d$-dimensions,
where two queens attack one another if they lie on a common
hyperplane.  It can then be described as the $n$-queens problem in
$d$-dimensions.  (The traditional 8-queens problem, as described
above, is 2-dimensional.)

The $n$-queens problem is classically considered a theoretical
one, but has also been studied~\cite{5} for its many applications:
in distributed memory storage schemes~\cite{4}, VLSI
testing~\cite{7}, deadlock prevention in computer
systems~\cite{6}, and others.  As a canonical problem in
constraint satisfaction, the problem is also approached using
neural networks~\cite{8}, and also studied as a standard candidate
for the backtracking (depth-first search) method.

It is trivial that \(\forall n \geq 4, d \geq 3\), there always is
a way of placing $n$ queens in a $d$-dimensional board of size $n$
so that no two attack each other.  This follows from the result
\cite{1} that for all \(n \geq 4\), $n$ queens can be placed on a
regular 2-dimensional board.

We previously performed a computational analysis~\cite{2} of the
$n$-queens problem in higher dimensions, by counting the number of
ways in which $n$ queens can be placed.  This analysis seemed to
indicate that there were non-attacking queens solutions in $d$
dimensions that could not be projected onto subspaces, and led us
to investigate the following question:

\begin{itemize}

\item[{\sc Least}] What is the least number of queens that would
be necessary to attack every position in a $d$-dimensional board
of size $n$?

\end{itemize}

Certain simple cases are easy to analyse.  For instance, if \(n =
3\), the least number of queens that can attack all board
positions on a $d$-dimensional board is 1, \(\forall d \geq 1\).
Similarly, when $d$ is 2, i.e., on the regular board of size $n$,
we obviously need no more than $n$ queens, and in fact, fewer
suffice.  Our present work gives the first general lower bound on
the number of queens needed, for all $n$ and $d$.

\section{The Queens Problem} \label{nqueens}

\subsection{In Two Dimensions}\label{twodim}

On a 2-dimensional grid, a queen can attack along the two axes $X$
and $Y$, and along two diagonals.  For a queen located at
Cartesian coordinates \(\langle q_1, q_2\rangle\), the axes are
given if \(x = q_1\) and \(y = q_2\).  The diagonals are \(x - q_1
= y - q_2\) and \(x - q_1 = q_2 - y\), and these are the 1- and
2-dimensional attack lines respectively.

\subsection{In $d$ Dimensions}\label{ddim}

Consider a $d$-dimensional hypercube.  The vertices of such a
hypercube can be addressed by bit strings of length $d$---the
$2^d$ vertices of the hypercube can be addressed by \(000\ldots0\)
through \(111\ldots1\).

Two vertices in a $d$-dimensional hypercube are adjacent if their
addresses differ in just one bit position.  If they differ in \(k
\leq d\) positions, they lie on a $k$-dimensional diagonal.
Therefore, the longest-length diagonal connects a vertex with its
polar reciprocal (the opposite vertex, whose address bits are all
inverses of the corresponding bits).

A chess board of size $n$ is similar to a hypercube, except that
instead of just two values, each position in the string takes $n$
values, from 0 through \(n - 1\).

A queen in a $d$-dimensional chess space therefore has a position
given by a $d$-dimensional vector \(\langle q_1, q_2, \ldots,
q_d\rangle\), where \(0 \leq q_k \leq n-1\).

The equations for the $d$-dimensional attack lines for a queen at
location \(\langle q_1, q_2, \ldots, q_d\rangle\) in a
$d$-dimensional chess hyperspace are of the form
\[\pm (x_1 - q_1) = \pm (x_2 - q_2) = \ldots = \pm (x_d - q_d).\]

There are obviously \(2^d\) such equations in all.  However,
noting that changing the signs on all the terms gives us a new
equation with the same meaning as the one changed, there are
\(2^{d-1}\) equations that are distinct.

\begin{lemma} \label{queenvectorslemma}

A queen in $d$-space can also be considered as having attack
vectors of the form \(\langle \delta_1, \delta_2, \ldots, \delta_d
\rangle\), \(\delta_k \in \{-1, 0, 1\}\), with the constraint that
\(\{\delta_1, \delta_2, \ldots, \delta_d\} \neq \{0\}\).

\end{lemma}

\begin{proof}

A component $\delta_k$ of an attack vector represents a queen's
movement in dimension $k$ along one attack line.  Such movement
can be in two directions (``forward'' or ``backward,'' so to
speak) which we can represent as $+1$ or $-1$; if there is no
movement in dimension $k$, then \(\delta_k = 0\).

It is not possible for all components to be zero, i.e.,
\(\{\delta_1, \delta_2, \ldots, \delta_d\} \neq \{0\}\), because
an attack vector must have the queen moving along at least one
dimension---a zero vector denotes complete lack of movement from 
the queen's current position.
\end{proof}

Each attack line for a queen is composed of the maximum range of
the queen's movement along two attack vectors (one in each
direction from the queen's location) such that their sum results
in the zero vector \(\langle 0_1, 0_2, \ldots, 0_d \rangle\).
Hence, the number of attack lines is half the number of attack
vectors.

\begin{remark} \label{remarkattacklines}
The board positions at a scalar distance \(s > 0\) along a queen's
attack vector are given by:

\[ s \langle \delta_1, \delta_2, \ldots, \delta_d \rangle + \langle q_1, q_2,
\ldots, q_d \rangle. \]
\end{remark}

\begin{theorem} \label{theoremattacklines}

A queen in $d$-space has a total of \[\frac{3^d - 1}{2}\] attack
lines.

\end{theorem}

\begin{proof}

We know from Lemma~\ref{queenvectorslemma} that each coordinate of
a queen's attack vector can have one of 3 values.  An attack
vector also has $d$ components (as it is a vector in $d$-space).
However, the zero vector where the queen does not move in any
direction is ruled out, as pointed out in the lemma.

Therefore, the queen has \(3^d - 1\) attack vectors.  Since the
number of attack lines is half the number of attack vectors, the
number of attack lines is \(\frac{3^d - 1}{2}\). \end{proof}

\subsection{{\sc Least}}

For the question {\sc Least} stated in Section~\ref{intro}, we first
consider the following.

\begin{lemma} \label{leastlemma}

A queen in any $d$-dimensional chess space of size $n$ can attack
at most $n$ board positions (including the one it holds) along any
one attack line.

\end{lemma}

\begin{proof}

By Remark~\ref{remarkattacklines}, we can find a board position
attacked by the queen along an attack vector by taking the sum of
queen's position and the product of the attack vector as given in
Lemma~\ref{queenvectorslemma} and a scalar distance \(s > 0\) from the
queen's position.

The individual $k$-dimensional coordinate of a board position
attacked by a queen is given by \(s \delta_k + q_k\).  We know
that the chess space itself does not have any position with a
coordinate greater than \(n-1\) or less than \(0\).  Therefore,
any attack vector terminates when \(s \delta_k + q_k = 0\) or \(s
\delta_k + q_k = n - 1\), for any coordinate $k$.  Therefore, $s$
can take at most $n$ values since at least one of the $\delta_k$
must be non-zero and it can only have $n$ values along that
coordinate. \end{proof}

From this and Theorem~\ref{theoremattacklines}, we get the
following.

\begin{lemma}

A queen in a $d$-dimensional chess space of size $n$ can attack no
more than \(\frac{n(3^d - 1)}{2}\) board positions.

\end{lemma}

Note also that there are \(n^d\) board positions in a
$d$-dimensional chess space of size $n$.  Therefore, by dividing
the number of board positions by the maximum number of positions
attacked by a queen, we have the following.

\begin{theorem} \label{leastqueens}

The least number of queens necessary to attack all positions in a
$d$-dimensional chess space of size $n$ is no less than

\[ \frac{2 n^{d - 1}}{3^d - 1}. \]

\end{theorem}

In deriving this result, we have assumed that every queen is able
to attack as many board positions as possible along every attack
line, and that no two queens attack the same board position.  This
is obviously an overestimate, so the above expression is a lower
bound subject to refinement using more intricate analyses.

No more than $n$ non-attacking queens can be placed on a
two-dimensional board (since there are only $n$ rows or columns,
and every queen must be on a separate row and column).  However,
Theorem~\ref{leastqueens} shows us that this is not true in
higher-dimensional spaces.  Specifically, we have the following.

\begin{corollary}

When \(d \geq 3\), it is not always possible to attack all board
positions on a board of size \(n > 3\) using $n$ queens.

\end{corollary}

\begin{proof}

Given Theorem~\ref{leastqueens}, we know that for any values $n$
and $d$ such that

\[ \frac{2 n^{d - 1}}{3^d - 1} > n, \]

which gives, upon simplification,

\[ 2 n^{d-2} > 3^d - 1, \]

it is impossible to attack all board positions using just $n$
queens. \end{proof}

Similarly, we can show that it is always possible to find an $n$
large enough that given a certain $d$, it is not possible to
attack all board positions using \(n^k\) queens, where \(k < d -
1\):

\begin{corollary}

If \(2 n^{d - k - 1} > 3^d - 1\), it is not possible to attack all
board positions using \(n^k\) queens.

\end{corollary}

It is therefore also possible to always find a chess space of large
enough dimension and size that for any $k$, \(n^k\) queens do not
suffice; alternatively, for any \(d > 2\), \(n^{d-2}\) queens do not
suffice for all but a finite number of $n$.

\section{Suggestions For Further Work}

Based on the work presented here, the authors see a likely future
result being the derivation of an exact expression for the number of
queens necessary to attack all positions in a $d$-dimensional board of
size $n$.  This may prove to be difficult, however, so an easier
effort should be one directed towards a better lower bound than
Theorem~\ref{leastqueens} provides.  Likewise, a non-trivial upper
bound (\(\frac{n^{d-1}}{d}\) being trivial) should also be obtainable.

Related to these, of course, are the more standard problems relating
to the enumeration of the possible solutions in case of a
$d$-dimensional board of size $n$.  Algorithms for placement of queens
in higher-dimensional spaces (a trivial problem when there are no more
queens than indicated by Theorem~\ref{leastqueens}) are also worth
investigating.  Considering the slew of applications of the standard
2-dimensional problem, it also remains to be seen what applications
can be made of the higher-dimensional analogue and results therein.

\section*{Acknowledgement}

The authors would like to thank K.~R.~Knopp for useful discussions
on this topic.

\end{document}